\documentclass[11pt,a4paper]{article}
\usepackage[english]{babel}
\usepackage{pst-grad} 
\usepackage{pst-plot} 
\usepackage{pstricks}
\usepackage{amsmath,amssymb,mathrsfs,amsthm}
\usepackage[dvips]{graphicx}
\usepackage{epsfig}

\newcommand{\RR}{\mathbb R}

\newcommand{\TT}{\mathbb T}
\newcommand{\pat}{\partial_t}
\newcommand{\pax}{\partial_x}

\newcommand{\jeps}{\mathcal{J}_\epsilon*}

\newcommand{\rh}{\ensuremath{ \langle\rho\rangle }}

\newcommand{\reps}{\rho_\epsilon}
\newcommand{\veps}{v_\epsilon}
\newcommand{\rhh}{\ensuremath{ \langle\rho_0\rangle }}

\usepackage[pdfauthor={Rafael Granero-Belinchon},pdftitle={...},bookmarks,colorlinks]{hyperref}

\newcounter{comentcount}
\newcounter{teocount}
\newtheorem{lem}{Lemma}
\newtheorem{prop}{Proposition}

\newtheorem{teo}[teocount]{Theorem}  
\newtheorem{defi}{Definition}

\newenvironment{remark}
{\stepcounter{comentcount} {\bf \tt Remark} {\bf\tt\arabic{comentcount}} }{ }
\title{An aggregation equation with a nonlocal flux}
\author{Rafael Granero-Belinch\'on$^{\mbox{{\footnotesize 1}}}$ and Rafael Orive Illera $^{\mbox{{\footnotesize 2}}}$}

\begin{document}

\maketitle 
\footnotetext[1]{Email: \texttt{rgranero@math.ucdavis.edu}\\
Department of Mathematics,\\
University of California, Davis,\\
CA 95616, USA}

\footnotetext[2]{Email: \texttt{rafael.orive@icmat.es}\\
Universidad Aut\'onoma de Madrid\\
Instituto de Ciencias Matem\'aticas (CSIC-UAM-UC3M-UCM)\\
C/Nicol\'as Cabrera, 13-15,
Campus de Cantoblanco,
28049 - Madrid (Spain)}

\vspace{0.3cm}

\begin{abstract}
In this paper we study an aggregation equation with a general nonlocal flux. We study the local well-posedness and some conditions ensuring global existence. We are also interested in the differences arising when the nonlinearity in the flux changes. Thus, we perform some numerics corresponding to different convexities for the nonlinearity in the equation.
\end{abstract}

\vspace{0.3cm}


\textbf{Acknowledgments}: The authors are supported by the Grants MTM2011-26696 and SEV-2011-0087 from Ministerio de Ciencia e Innovaci\'on (MICINN).

\tableofcontents

\section{Introduction}

In this paper we study several types of nonlinear and nonlocal aggregation models with nonlinear diffusion and self-attraction coming from the Poisson equation posed in a periodical setting, i.e., the spatial domain is $\TT=[-\pi,\pi]$. In particular we are interested in the differences arising when the nonlinearity in the diffusion changes. 

Let us consider $\beta(x)$ a smooth, positive function in a domain containing $\RR^+$. The problem reads
\begin{equation}\label{eq2}
\left\{
\begin{array}{l}
\pat\rho=\pax\left(-\beta(\rho)H\rho+\rho\pax v\right),\qquad x\in\TT, t>0,\\
\pax^2 v=\rho-\rh,
\end{array}
\right.
\end{equation}
where $H\rho$ denotes the periodic Hilbert Transform 
$$
H\rho(x)=\frac{1}{2\pi}\text{P.V.}\int_\TT \frac{\rho(y)}{\tan\left(\frac{x-y}{2}\right)}dy,
$$
and 
$$
\rh=\frac{1}{2\pi}\int_\TT\rho(x)dx.
$$
In our favourite units, the scalar $v$ is the gravitational potential. Clearly, we need to attach an smooth initial data, $\rho_0$, that we will take non-negative. 

In (\ref{eq2}) we generalize  the classical Smoluchowski-Poisson or Patlak-Keller-Segel system  considering a quasilinear and critical nonlocal diffusion (\emph{i.e.}, the case where the diffusion is given by $\sqrt{-\Delta}$).  
Up to the best of our knowledge, this equation  has not been studied before. However, the one-dimensional case with linear, nonlocal diffusion has been treated in \cite{AGM, bournaveas2010one}, while similar equations with linear and quasilinear, local diffusions have been considered in many works (see for instance \cite{BN, BNB, blanchet2011parabolic, blanchet2009critical, BCM, 
blanchet2010functional, Dolbeault4, Dolbeault3, Dolbeault2, Dolbeault} and the references therein). In particular, the linear and local diffusion counterpart of (\ref{eq2}) is
\begin{equation}\label{eq2local}
\left\{
\begin{array}{l}
\pat\rho=\beta\pax^2\rho+\pax\left(\rho\pax v\right),\qquad x\in\TT, t>0,\\
\pax^2 v=\rho-\rh,
\end{array}
\right.
\end{equation}
with $\beta>0$. This system has been previously addressed as a model of gravitational collapse by Biler and collaborators (see \cite{Bi5,Bi4,Bi3}), and Chavanis and Sire (see \cite{Chd,Cha}). Also, the system \eqref{eq2local} has been also proposed as a model of chemotaxis in biological system (see \cite{jager1992explosions, keller1970initiation, patlak1953random}).  We also notice that, in two dimensions, \eqref{eq2local} can be re-written as
$$
\left\{
\begin{array}{l}
\pat \rho=\beta\Delta\rho+\nabla\cdot(\nabla U\rho),\\
\Delta U=\rho-\rh,
\end{array}
\right.
$$
which is similar to the vorticity formulation for 2D Navier-Stokes. Its mathematical properties have been widely studied in different physical contexts (see e.g. \cite{Bi6,blanchet2011parabolic, BCM, blanchet2010functional, Dolbeault4,Dolbeault3, Carrillo2, Dolbeault2, Dolbeault} and references therein. 
For instance, Corrias, Perthame and Zaag in \cite{corrias2004global} proved that, for small data $\rho_0$ in $L^{d/2}$, where $d\geq 2$ is the space dimension, there are global in time weak solutions to equation \eqref{eq2local}  and blow-up if the smallness condition does not hold. Global existence when the initial data is small in $L^1$ has been recently addressed in \cite{nagai2011global}. The case of measure-valued weak solutions has been considered by Senba and Suzuki  in \cite{senba2002weak}.

Its linear and nonlocal diffusion counterpart is 
\begin{equation}\label{eq2nolocal}
\left\{
\begin{array}{l}
\pat\rho=-\beta\Lambda^\alpha\rho+\pax\left(\rho\pax v\right),\text{ for $x\in\TT, t>0$},\\
\pax^2 v=\rho-\rh,
\end{array}
\right.
\end{equation}
where the fractional Laplacian $\Lambda^\alpha=(-\Delta)^{\alpha/2}$ is defined using Fourier techniques as follows
$$
\widehat{\Lambda^\alpha u}=|\xi|^\alpha\hat{u}.
$$
This non-local generalization \eqref{eq2nolocal} has been recently studied (see \cite{biler2010blowup,BilerWu,bournaveas2010one}). In particular, Li, Rodrigo and Zhang \cite{li2010exploding} have established local existence, a continuation criterion and the existence of finite time singularities for the two dimensional case.  In \cite{AGM} Ascasibar, Granero and Moreno have recovered the local existence and the continuation criterion by means of different techniques. Also, the global existence for small initial datum in $L^\infty$ for all $0<\alpha<2$ and $d=2,3$ and $\alpha=1$ while $d=1$, and global existence for $1<\alpha<2$ and $d=1$ has been proved in \cite{AGM}.

Nonlinear generalizations of \eqref{eq2local} have been studied in \cite{BN, BNB, blanchet2009critical} and the references therein, while nonlinear generalizations of \eqref{eq2nolocal} have been addressed in \cite{CV, CC, CC2, CCCF}. In particular, in \cite{Carrillo, CC, CC2, CCCF}, the authors studied the equation
\begin{equation}\label{eq2angel}
\pat\theta=\pax\left(-\theta H\theta\right)-\nu\Lambda\theta.
\end{equation}
This equation has been proposed as a one-dimensional model of the 2D Vortex Sheet
problem or a one dimensional model of the 2D surface quasi-geostrophic equation. Some of its mathematical properties are well-known. In particular, Castro and C\'ordoba proved  local existence, global existence with assumptions, blow-up in finite
time and ill-posedness depending on the sign of the initial data for classical solutions of equations \eqref{eq2angel}. We notice that \eqref{eq2} without self-attraction terms and $\beta(x)=x$ reduces to \eqref{eq2angel}.

The plan of the paper is as follows: in Section \ref{sec2} we prove that the problem is well-posed. In Section \ref{sec3} we obtain a uniform bound for $\|\rho(t)\|_{L^\infty(\TT)}$. In Sections \ref{sec4} and \ref{sec5} we prove global existence of solution corresponding to small initial data in $H^2$ and $L^1$ for some choices of $\beta(x)$. Finally, in Section \ref{sec6} we perform some numerical simulations to better understand the role of $\beta$.

\section{Well-posedness}\label{sec2}
In this section we study the existence of classical solution to \eqref{eq2} in a small time interval $[0,\tau]$ with initial data  $\rho_0(x)\geq0$ in the Sobolev space $H^2(\TT)$ with the natural norm defined by
$$
\|\rho\|_{H^s(\TT)}^2=\|\rho\|_{L^2(\TT)}^2+\|\Lambda^s \rho\|_{L^2(\TT)}^2.
$$
The main ingredients of the proof are the following identities 
\begin{equation}\label{calderon}
\frac{1}{2\pi}\|Hg\|^2_{L^2(\TT)}=\frac{1}{2\pi}\|g\|^2_{L^2(\TT)}-\langle g \rangle^2\leq \frac{1}{2\pi}\|g\|^2_{L^2(\TT)},
\end{equation}
and, for general $1<p<\infty$,
\begin{equation}\label{calderon2}
\|Hg\|_{L^p(\TT)}\leq C_H(p)\|g\|_{L^p(\TT)},
\end{equation}
see \cite{stein1970singular} for further details on singular integral operators.

We note that $\Lambda g=H\pax g=\sqrt{-\pax^2}\,g$. This operator has the following integral representation  
$$
\Lambda g(x)=\frac{1}{2\pi}\text{P.V.}\int_\TT\frac{g(x)-g(y)}{\sin^{2}\left(\frac{x-y}{2}\right)}dy.
$$
We will require the following pointwise (see \cite{cordoba2003pointwise}) 
\begin{equation}\label{intporpart}
2g\Lambda g\geq \Lambda(g^2).
\end{equation}
Also, we use the following Gagliardo-Niremberg-Sobolev inequalities
\begin{eqnarray}\label{GN}
\|\pax \rho\|_{L^4(\TT)}&\leq & \sqrt{3}\|\pax^2\rho\|_{L^2(\TT)}^{1/2}\|\rho-\langle \rho\rangle\|_{L^\infty(\TT)}^{1/2}\\
\label{S1}
\|\rho-\rh\|_{L^\infty(\TT)}&\leq & C_S\|\pax\rho\|_{L^2(\TT)}.
\end{eqnarray}
Moreover, we can take $C_S$ such that
\begin{equation}\label{S2}
\|\rho-\rh\|_{C^{1/2}(\TT)}\leq C_S\|\pax\rho\|_{L^2(\TT)}.
\end{equation}
Using the reverse triangle inequality, we get 
\begin{equation}\label{S}
\|\rho\|_{L^\infty(\TT)}\leq C_S\|\pax\rho\|_{L^2(\TT)}+\rh.
\end{equation}

To prove the existence and uniqueness of classical solution, we proceed as in \cite{bertozzi-Majda}. First, we obtain some \emph{`a priori'}  bounds for the usual norm in the space $H^k$. Then, we regularize equation \eqref{eq2} and prove that all the regularized systems have a classical solution for a  small time $\tau(\rho_0)$. To conclude, we use the \emph{`a priori'} bound to show that the solutions to the regularized problem form a Cauchy sequence whose limit is the solution to the original equation. The result is

\begin{teo}[Local well-posedness]\label{WellPT}
Let $\beta\in C^4([0,\infty))$, $\beta(x)\geq 0$ be a given function. Let $\rho_0\in H^k(\TT)$ with $k\geq 2$ and $\rho_0\geq 0$ be the initial data. Then, there exists an unique solution $\rho\in C([0,\tau],H^k(\TT))$ of \eqref{eq2} with $\tau=\tau(\rho_0)>0$. Moreover, $$\rho\in C^1([0,\tau],C(\TT))\cap C([0,\tau],C^{1,1/2}(\TT)).$$
\end{teo}

In order to simplify the notation, we will abbreviate $\rho(x,t) = \rho(x)$, or simply $\rho$, throughout the rest of the paper.

\begin{proof}
First, we remark that, for nonnegative initial data, the solution remains nonnegative and we have conservation of mass
$$
\int \rho(x,t)dx=\int \rho(x,0)dx.
$$
Thus, $\rh=\langle\rho_0\rangle$ is a constant depending only on the initial data. We show the case $k=2,$ being the other cases analogous. Now, fix $\lambda >\|\rho_0\|_{L^\infty}$ a constant and define the energy
\begin{equation}\label{enerho}
E[\rho]=\|\rho\|_{H^2(\TT)}+\left\|d[\rho]\right\|_{L^\infty(\TT)},
\end{equation}
where
\begin{equation}\label{drho}
d[\rho]=\frac{1}{\lambda-\rho(x)}.
\end{equation}
Due to the smoothness of $\beta$, for $0\leq x\leq\lambda$, we have
\begin{multline*}
|\pax^j\beta(x)|-|\pax^j \beta(0)|\leq\left||\pax^j\beta(x)|-|\pax^j \beta(0)|\right|\\
\leq|\pax^j\beta(x)-\pax^j \beta(0)|\leq \max_{0\leq y\leq \lambda}|\pax^{j+1}\beta(y)| \, x,
\end{multline*}
and
\begin{equation}\label{boundbeta}
|\pax^j\beta(x)|\leq c(\beta,\lambda)(x+1),\;\;j=0,...,3.
\end{equation}
Now we study the evolution of the $L^2(\TT)$ norm of the solution. We denote $c$ a constant depending only on the function $\beta$ and on the fixed constant $\lambda$. Thus, this constant is harmless and it can change from line to line. Using \eqref{calderon} and H\" older's inequality, we have
\begin{multline*}
\frac{1}{2}\frac{d}{dt}\|\rho\|_{L^2(\TT)}^2
=\int_\TT\pax\rho(x)\beta(\rho(x))H\rho(x)dx\\+\int_{\TT}\rho(x)\pax\rho(x)\pax v(x)dx+\int_{\TT}\rho^2(x)(\rho-\rh)dx\\
\leq c\|\rho\|_{L^2(\TT)}\left((\|\rho\|_{L^ \infty (\TT)}+1)\|\pax \rho\|_{L^2(\TT)}+\|\rho-\langle\rho\rangle\|_{L^\infty(\TT)}\|\rho\|_{L^2(\TT)}\right).
\end{multline*}
Using Sobolev embedding \eqref{S} and the inequality 
$$
\|\rho-\langle\rho\rangle\|_{L^\infty(\TT)}\leq\|\rho\|_{L^\infty(\TT)},
$$ 
we have
\begin{equation}\label{L2}
\frac{d}{dt}\|\rho\|_{L^2(\TT)}\leq c(\|\rho\|_{L^\infty(\TT)}+1)\|\rho\|_{H^2(\TT)}\leq c(\|\rho\|_{H^2(\TT)}+1)^2.
\end{equation}
We study now the second derivative. Firstly, the transport terms corresponding to $v$:
\begin{multline}\label{I1}
I_1=\int_\TT\pax^ 2\rho\pax^ 2(\pax\rho\pax v + \rho(\rho-\rh))\\
\leq\int_\TT\pax^ 2\rho(\pax^ 3\rho\pax v+3\pax^2\rho(\pax\rho)^2+3 \pax^ 2\rho(\rho-rh)+ \pax^ 2\rho\rho)dx\\
\leq c(\|\rho-\rh\|_{L^\infty(\TT)}+\|\rho\|_{L^\infty(\TT)})\|\pax^ 2\rho\|_{L^2(\TT)}\|\rho\|_{H^2(\TT)},
\end{multline}
where in the last step we have used \eqref{GN}. Now, the term corresponding to the nonlinear diffusion is
\begin{eqnarray*}
I_2 &=&-\int_\TT\pax^ 2\rho\pax^ 2(\beta(\rho)\Lambda \rho)\\
&=&-\int_\TT\pax^ 2\rho(\beta''(\rho)(\pax\rho)^2\Lambda\rho+\beta'(\rho)\pax^2\rho\Lambda\rho+2\beta'(\rho)\pax\rho\Lambda\pax\rho+\beta(\rho)\Lambda\pax^ 2\rho)dx.
\end{eqnarray*}
Using \eqref{GN} and \eqref{boundbeta}, we obtain
\begin{multline*}
J_1 =-\int_\TT\pax^ 2\rho\beta''(\rho)(\pax\rho)^2\Lambda\rho dx\\
\leq c\left(\|\rho\|_{L^\infty(\TT)}+1\right)\|\Lambda \rho\|_{L^\infty(\TT)}\|\rho-\rh\|_{L^\infty(\TT)}\|\pax^2\rho\|_{L^2(\TT)}^2,
\end{multline*}
and
$$
J_2=-\int_\TT\beta'(\rho)(\pax^2\rho)^2\Lambda\rho dx\leq c\left(\|\rho\|_{L^\infty(\TT)}+1\right)\|\Lambda\rho\|_{L^\infty(\TT)}\|\pax^2\rho\|_{L^2(\TT)}^2.
$$
Due to \eqref{calderon}, we get
$$
J_3=-\int_\TT\beta'(\rho)\pax\rho\pax^2\rho\Lambda\pax\rho dx\leq c\left(\|\rho\|_{L^\infty(\TT)}+1\right)\|\pax \rho\|_{L^\infty(\TT)}\|\pax^2\rho\|_{L^2(\TT)}^2. 
$$
We study now the singular term  in $I_2$. By \eqref{intporpart}, we have
$$
J_4=-\int_\TT\pax^ 2\rho\beta(\rho)\Lambda\pax^ 2\rho dx\leq -\frac{1}{2}\int_\TT\beta(\rho)\Lambda(\pax^2\rho)^ 2 dx\leq -\frac{1}{2}\int_\TT\Lambda\beta(\rho)(\pax^2\rho)^ 2 dx.
$$
We compute 
\begin{multline*}
\Lambda \beta(\rho(x))=\frac{1}{2\pi}\text{P.V.}\int_\TT\left(\frac{\beta(\rho(x))-\beta(\rho(y))}{\rho(x)-\rho(y)}-\beta'(\rho(x))\right)\frac{\rho(x)-\rho(y)}{\sin^{2}\left(\frac{x-y}{2}\right)}dy\\+\beta'(\rho(x))\Lambda \rho(x).
\end{multline*}
Using Taylor's Theorem, we obtain
$$
\left|\frac{\beta(\rho(x))-\beta(\rho(y))}{\rho(x)-\rho(y)}-\beta'(\rho(x))\right|\leq \frac{|\beta''(\zeta)||\rho(x)-\rho(y)|}{2}.
$$
Since we have an extra cancellation and using \eqref{boundbeta}, we have the bound
\begin{equation}\label{boundlambda}
|\Lambda \beta(\rho(x))|\leq c(\|\rho\|_{L^\infty(\TT)}+1)\left(\|\pax \rho\|^2_{L^\infty(\TT)}+ \|\Lambda \rho\|_{L^\infty(\TT)}\right).
\end{equation}
Putting all together we obtain
$$
J_4\leq \|\pax^ 2\rho\|_{L^2(\TT)}^2c(\|\rho\|_{L^\infty(\TT)}+1)\left(\|\pax \rho\|^2_{L^\infty(\TT)}+ \|\Lambda \rho\|_{L^\infty(\TT)}\right).
$$
Collecting all the estimates and using Sobolev embedding, we obtain
\begin{equation}\label{I2}
I_2=J_1+J_2+J_3+J_4\leq c\|\pax^ 2\rho\|_{L^2(\TT)}(\|\rho\|_{H^2(\TT)}+1)^4.
\end{equation}
It remains to show that the transport terms with a singular non-local velocity 
$$
I_3=-\int_\TT\pax^ 2\rho\pax^ 2(\beta'(\rho)\pax\rho H\rho),
$$
are bounded. The lower order terms can be bounded as follows
\begin{multline*}
J_5=-\int_\TT\pax^2\rho\beta'''(\rho)(\pax\rho)^ 3H\rho dx\\
\leq c\left(\|\rho\|_{L^\infty(\TT)}+1\right)\|\pax \rho\|_{L^\infty(\TT)}\|\rho-\rh\|_{L^\infty(\TT)}\|H \rho\|_{L^\infty(\TT)}\|\pax^2\rho\|_{L^2(\TT)}^2,
\end{multline*}
\begin{multline*}
J_6=-\int_\TT(\pax^2\rho)^2\beta''(\rho)\pax\rho H\rho dx\leq c\left(\|\rho\|_{L^\infty(\TT)}+1\right)\|\pax \rho\|_{L^\infty(\TT)}\\
\cdot\|H \rho\|_{L^\infty(\TT)}\|\pax^2\rho\|_{L^2(\TT)}^2,
\end{multline*}
\begin{multline*}
J_7=-\int_\TT\pax^2\rho\beta''(\rho)(\pax\rho)^2 \Lambda\rho dx\leq c\left(\|\rho\|_{L^\infty(\TT)}+1\right)\|\Lambda \rho\|_{L^\infty(\TT)}\\
\cdot\|\rho-\rh\|_{L^\infty(\TT)}\|\pax^2\rho\|_{L^2(\TT)}^2,
\end{multline*}
\begin{multline*}
J_8=-\int_\TT\pax^2\rho\beta'(\rho)(\Lambda\rho \pax^2\rho+\Lambda\pax\rho \pax\rho) dx\\
\leq c\left(\|\rho\|_{L^\infty(\TT)}+1\right)(\|\Lambda \rho\|_{L^\infty(\TT)}+\|\pax \rho\|_{L^\infty(\TT)})\|\pax^2\rho\|^2_{L^2(\TT)}.
\end{multline*}
The most singular term in $I_3$ is
\begin{multline*}
J_9=-\int_\TT\pax^2\rho\beta'(\rho)H\rho \pax^3\rho dx\\
\leq c\left(\|\rho\|_{L^\infty(\TT)}+1\right)(\|\Lambda \rho\|_{L^\infty(\TT)}+\|\pax \rho\|_{L^\infty(\TT)}\|H \rho\|_{L^\infty(\TT)})\|\pax^2\rho\|^2_{L^2(\TT)}.
\end{multline*}
Thus, we obtain the following bound
\begin{equation}\label{I3}
I_3=J_5+\cdots +J_9\leq c\|\pax^ 2\rho\|_{L^ 2(\TT)}(\|\rho\|_{H^2(\TT)}+1)^5.
\end{equation}
Then,
$$
\frac{1}{2}\frac{d}{dt}\|\pax^ 2\rho\|_{L^ 2(\TT)}^ 2\leq c\|\pax^ 2\rho\|_{L^ 2(\TT)}(\|\rho\|_{H^2(\TT)}+1)^5,
$$
and, using \eqref{L2}, \eqref{I1}, \eqref{I2} and \eqref{I3}, we conclude that, while $\|\rho\|_{L^\infty(\TT)}<\lambda$, the following inequality holds
\begin{equation}\label{boundh}
\frac{d}{dt}\|\rho\|_{H^2(\TT)}\leq c(\|\rho\|_{H^2(\TT)}+1)^5\leq c(E[\rho]+1)^5 .
\end{equation}
We need a bound for the remaining term in the energy \eqref{enerho}. Using \eqref{enerho}, \eqref{drho} and Sobolev embedding, we have
$$
\frac{d}{dt}d[\rho]\leq d[\rho]^2\|\pat\rho\|_{L^\infty(\TT)}\leq d[\rho]\|d[\rho]\|_{L^ \infty(\TT)} E[\rho]^3.
$$
Thus, we obtain
$$
d[\rho](t+h)\leq d[\rho](t)\exp
\left(\int_t^{t+h} \|d[\rho]\|_{L^ \infty(\TT)} E[\rho]^3ds\right).
$$
Finally, we have
\begin{equation}\label{boundd}
\frac{d}{dt}\|d[\rho]\|_{L^\infty(\TT)}=\lim_{h\rightarrow0}\frac{\|d[\rho](t+h)\|_{L^\infty(\TT)}-\|d[\rho](t)\|_{L^\infty(\TT)}}{h}\leq E[\rho]^ 5.
\end{equation}
Thanks to \eqref{boundh} and \eqref{boundd} we conclude the \emph{'a priori'} energy estimates:
$$
\frac{d}{dt}E[\rho]\leq c(E[\rho]+1)^5,
$$
and then
\begin{equation}\label{boundE}
E[\rho](t)\leq \frac{E[\rho_0]+1}{\sqrt[4]{1-4ct(E[\rho_0]+1)^4}}-1.
\end{equation}

Our next step is classical. We consider $\mathcal{J}$ a symmetric and positive mollifier, see \cite{bertozzi-Majda}. For $\epsilon>0$, we define
\begin{equation}
\mathcal{J}_\epsilon(x)=\frac{1}{\epsilon}\mathcal{J}\left(\frac{x}{\epsilon}\right)
\label{epsi} 
\end{equation}
and consider the regularized problems
$$
\left\{
\begin{array}{l}
\pat \rho^{\epsilon}=\jeps\pax\left(-(\beta(\jeps\rho^{\epsilon})+\epsilon)H\jeps\rho^{\epsilon}\right)+\jeps\pax\left(\jeps\rho^\epsilon\pax v^\epsilon\right),\\
\pax^ 2 v^\epsilon=\jeps\rho^{\epsilon}-\rh.
\end{array}
\right.
$$
Notice that these regularized systems remains positive for all times and conserve the total mass,
$$
\|\rho^\epsilon(t)\|_{L^1}=\|\rho_0\|_{L^1}.
$$
Thus, using Tonelli's Theorem in a classical way, we get
$$
\|\jeps\rho^\epsilon(t)\|_{L^1}=\|\rho_0\|_{L^1}.
$$
We can apply Picard's Theorem to these regularized problems. Define the set 
$$
O_\varsigma^\sigma=\{\rho\in H^2(\TT), \|\rho\|_{H^2(\TT)}<\sigma,\|\rho\|_{L^ \infty(\TT)}<\varsigma\},
$$
with $\|\rho_0\|_{H^2(\TT)}<\sigma$ and $\|\rho_0\|_{L^\infty(\TT)}<\varsigma<\lambda$, and observe that it is a non-empty open set in $H^2(\TT)$. To prove this claim just observe that, due to the Sobolev embedding, $\|\cdot\|_{L^ \infty(\TT)}$ and $\|\cdot\|_{H^ 2(\TT)}$ are continuous functionals. In this set we have $E[\rho]\leq C(\sigma,\lambda,\beta)$. Then, there exists a sequence $\rho^\epsilon$ of solutions to the regularized problems. For each $\rho^ \epsilon$ the bound \eqref{boundE} is also valid. So, we have a common time interval $[0,\tau(\rho_0)]$ where the solutions live. Now we can pass to the limit $\epsilon\to0$. To show this claim we have to prove that the sequence is Cauchy in the lower norm $L^ 2(\TT)$. These steps are quite classical, so, for the sake of brevity, we left the details for the interested reader. This concludes with the existence issue. 

We need to prove the uniqueness. Suppose that $\rho_1,\rho_2$ are two different classical solutions corresponding to the same initial datum and denote $\rho=\rho_1-\rho_2$. Then, 
\begin{multline*}
\frac{1}{2}\frac{d}{dt}\|\rho\|^ 2_{L^2(\TT)}=\int_\TT\rho\left(\rho_1(\rho_1-\rhh)-\rho_2(\rho_2-\rhh)+\pax\rho_1\pax v_1-\pax\rho_2\pax v_2\right.\\
\left.-(\beta(\rho_1)\Lambda\rho_1-\beta(\rho_2)\Lambda\rho_2)-(\beta'(\rho_1)\pax\rho_1H\rho_1-\beta'(\rho_2)\pax\rho_2H\rho_2)\right)dx.
\end{multline*}
We compute
\begin{multline*}
I_4=\int_\TT\rho\left(\rho_1(\rho_1-\rh)-\rho_2(\rho_2-\rh)\right)dx \\
=\int_\TT\rho^2\left(\rho_1+\rho_2-\rhh\right)\leq c(\|\rho_1\|_{H^2(\TT)},\|\rho_2\|_{H^2(\TT)})\|\rho\|_{L^2(\TT)}^2,
\end{multline*}
and
\begin{multline*}
I_5=\int_\TT\rho\left(\pax\rho\pax v_1+\pax\rho_2(\pax v_1-\pax v_2)\right)dx\\
\leq c(\|\rho_1\|_{H^ 2(\TT)},\|\rho_2\|_{H^ 2(\TT)})\|\rho\|_{L^2(\TT)}(\|\rho\|_{L^2(\TT)}+\|\pax(v_1-v_2)\|_{L^2(\TT)}).
\end{multline*}
Notice that we have
$$
\|\pax(v_1-v_2))\|_{L^2(\TT)}^2\leq\|v_1-v_2\|_{L^2(\TT)}\|\rho \|_{L^2(\TT)}.
$$
Using Poincar\' e inequality for $\|v_1-v_2\|_{L^2(\TT)}$, we get
$$
I_5\leq c(\|\rho_1\|_{H^ 2(\TT)},\|\rho_2\|_{H^ 2(\TT)})\|\rho\|_{L^2(\TT)}^2.
$$
Now we have to deal with the nonlinear diffusion:
$$
I_6=-\int_\TT\rho\left((\beta(\rho_1)-\beta(\rho_2))\Lambda\rho_1+\beta(\rho_2)\Lambda\rho\right)dx=J_{10}+J_{11}.
$$
In the term $J_{10}$ we use the smoothness of the function $\beta$ to obtain
$$
|\beta(\rho_1)-\beta(\rho_2)|\leq c(\|\rho_1\|_{H^ 2(\TT)},\|\rho_2\|_{H^ 2(\TT)})|\rho|,
$$ 
and we conclude 
$$
J_{10}\leq c(\|\rho_1\|_{H^ 2(\TT)},\|\rho_2\|_{H^2(\TT)})\|\rho\|_{L^2(\TT)}^2.
$$
In $J_{11}$ we use \eqref{intporpart} and \eqref{boundlambda} and we obtain a similar bound. We conclude
$$
I_6\leq c(\|\rho_1\|_{H^2(\TT)},\|\rho_2\|_{H^2(\TT)})\|\rho\|_{L^2(\TT)}^2.
$$
Only remains the transport term with the Hilbert Transform. We have
\begin{eqnarray*}
I_7&=&-\int_\TT\rho\left(\beta'(\rho_1)\pax\rho_1H\rho_1-\beta'(\rho_2)\pax\rho_2H\rho_2\right)dx,\\
&=&-\int_\TT\rho(\beta'(\rho_1)-\beta'(\rho_2))\pax\rho_1H\rho_1 dx-\int_\TT\rho \beta'(\rho_2)(\pax\rho H\rho_1+\pax\rho_2 H\rho))dx,\\
&\leq & c(\|\rho_1\|_{H^ 2(\TT)},\|\rho_2\|_{H^ 2(\TT)})\|\rho\|_{L^2(\TT)}^2,
\end{eqnarray*}
where we have used that $\beta$ is smooth enough. Then, collecting all the estimates together and using Gronwall's Inequality, we conclude the uniqueness.
\end{proof}

\section{Continuation criteria}\label{sec3}

In this section we use the following Lemma was proved in \cite{AGM} with inessential changes.

\begin{lem}\label{LemLambda}
Let $\rho\geq0$ be a smooth function that attains its maximum in the point $x_t$ and such that this maximum verifies $\rho(x_t)\geq4\rh$. Then
$$
\Lambda\rho\geq\frac{\rho^ 2(x_t)}{4\pi^ 2\rh}.
$$
\end{lem}

The following result study the absence of blow up for $\|\rho\|_{L^\infty(\TT)}$:

\begin{prop}\label{max}
Let $\rho$ be smooth solution of \eqref{eq2} under the hypothesis of Theorem~\ref{WellPT}. Let $\tau$ be the maximum lifespan of $\rho$. Assume that
\begin{equation}\label{cbeta}
\lim\limits_{\rho\rightarrow\infty}\beta(\rho)=\infty.
\end{equation} 
Then, the following inequality holds:
$$
\|\rho(t)\|_{L^ \infty(\TT)}\leq C\left(\rho_0,\beta\right)\quad \forall 0\leq t<\tau.
$$
\end{prop}
\begin{proof}Using the smoothness of $\rho$ we have that 
$$\|\rho(t)\|_{L^\infty(\TT)}=\max_{x\in\TT}\rho(x,t)=\rho(x_t),$$ is a Lipschitz function. We assume that $\rho(x_t)\geq4\rhh$. Then, applying Rademacher Theorem to the function $\rho(x_t)$ and Lemma \ref{LemLambda} (see \cite{cor2} for the details), the evolution of this quantity can be bounded as
$$
\frac{d}{dt}\|\rho(t)\|_{L^\infty(\TT)}< \left(1-\frac{\beta\left(\|\rho(t)\|_{L^\infty(\TT)}\right)}{2\pi\|\rho_0\|_{L^1(\TT)}}\right)\|\rho(t)\|^2_{L^\infty(\TT)}.
$$
Now it is enough to take and conclude
$$
C(\rho_0,\beta)=
\min_{\alpha\in\RR}\left\{\alpha\geq\max\{\|\rho_0\|_{L^\infty(\TT)},4\rhh\},\; \hbox{ such that }
\beta(\alpha)\geq 4\pi^2\rhh \right\}.
$$
\end{proof}

The proof of the following result is straightforward.

\begin{prop}\label{max2}
Let $\rho$ be the smooth solution of \eqref{eq2} under the hypothesis of Theorem~\ref{WellPT} and $\tau$ be the maximum lifespan of $\rho$. We assume that the initial data satisfies
$$
\|\rho_0\|_{L^1(\TT)}< \frac{\nu}{2\pi},
$$
and the function $\beta$ satisfies
\begin{equation}\label{cbeta2}
\beta(\rho)\geq \nu\quad\hbox{ if }\rho\geq R\quad\hbox{ for some constants }\nu,R>0.
\end{equation} 
Then, the following inequality holds:
$$
\|\rho(t)\|_{L^ \infty(\TT)}\leq R, \quad \forall 0\leq t<\tau.
$$
\end{prop}
 
As a consequence of the energy estimates we obtain a continuation criteria akin to the well-known Beale-Kato-Majda criterion in fluid dynamics \cite{beale1984remarks}: 

\begin{teo}[Continuation criteria]\label{CC} Let $\rho$ be a smooth solution in $(0,T)$ of \eqref{eq2} under the hypothesis of Theorem~\ref{WellPT} and $\beta$ satisfies \eqref{cbeta}. Then, if  
$$
\int_0^T\|\pax\rho(s)\|_{L^\infty(\TT)}^2+\|\Lambda\rho(s)\|_{L^\infty(\TT)}ds<\infty,
$$ 
the classical solution exists in $0\leq t\leq T+\delta$ for some $\delta>0$.
\end{teo}

\begin{proof}
Using the energy estimates in Theorem \ref{WellPT} and the boundedness of $\|\rho\|_{L^\infty(\TT)}$ we obtain the following
$$
\frac{d}{dt}\|\rho\|_{H^2(\TT)}\leq c(\rho_0,\beta)\|\rho\|_{H^ 2(\TT)}Q(\rho),
$$
where 
$$
Q(\rho)=\|\pax \rho\|_{L^\infty(\TT)}\left(1+\|H \rho\|_{L^\infty(\TT)}\right)+\|\Lambda \rho\|_{L^\infty(\TT)}+\|\Lambda \beta(\rho)\|_{L^\infty(\TT)}.
$$
Using \eqref{boundlambda}, the properties of the Hilbert transform and the finiteness of the domain, we get
$$
Q(\rho)\leq c\left(\|\Lambda \rho\|_{L^\infty(\TT)}+\|\pax \rho\|^2_{L^\infty(\TT)}+1\right).
$$
To conclude the result we use Gronwall inequality,
\begin{equation}\label{comentref}
\|\rho\|_{H^2(\TT)}(T)\leq \|\rho_0\|_{H^2(\TT)}e^{c(\rho_0,\beta)\left(T+\int_0^T\|\pax\rho(s)\|_{L^\infty(\TT)}^2+\|\Lambda\rho(s)\|_{L^\infty(\TT)}ds\right)}.
\end{equation}
\end{proof}
\begin{remark}
We remark that in the case of $\beta(\rho)\equiv\beta$ the continuation criteria is given by the condition
$$ 
\int_0^{T}\|\rho(s)\|_{L^ \infty(\TT)}ds<\infty,
$$
as was first proved in \cite{li2010exploding} for the 2D case and also in \cite{AGM} where it was obtained by means of a different method. The importance of this Theorem relies in its characterization of the possible finite time singularities. Indeed, let's assume that $\rho(x,t)$ is a solution showing finite time existence (up to time $T^*$). Then, using the previous result we conclude that 
$$
\limsup_{t\rightarrow T^{*}} \|\pax\rho(t)\|_{L^\infty(\TT)}+\|\Lambda\rho(t)\|_{L^\infty(\TT)}=\infty.
$$
\end{remark}
\begin{remark}
We note that a bound for $\delta$ can be obtained using \eqref{boundE} and \eqref{comentref}.
\end{remark}

\section{Global existence of classical solution for small initial data}\label{sec4}

In this section we show the existence of global solutions for small initial data in $H^2$ when the diffusion does not degenerate. The general case with initial data in $H^k$ is analogous.

\begin{teo} Let $\beta\in C^4[0,\infty)$ be a positive function satisfying \eqref{cbeta} and 
$$
\beta(\rho)\geq \nu,\quad\hbox{ for some }\nu>0.
$$
 Then, for all initial data $\rho_0\in H^2(\TT)$ satisfying 
$$
\|\rho_0\|_{L^1(\TT)}<2\pi\nu,\quad \|\pax^2\rho_0\|_{L^2(\TT)}\leq \mathcal{C},
$$
for an explicit constant $\mathcal{C}=\mathcal{C}\left(\beta,\rh\right)>0$ sufficiently small, there exists  a  solution of \eqref{eq2} such that
$$
\rho\in C([0,\infty),H^2(\TT)).
$$
\end{teo}
\begin{proof}
By Theorem \ref{WellPT}, there exists $\tau>0$ such that $\rho\in C([0,\tau],H^2(\TT))$. The idea is to strengthen the energy estimates. Since $\beta$ satisfy the hypothesis of Proposition~\ref{max}, we have the bound
$$
\|\rho(t)\|_{L^\infty(\TT)}\leq C\left(\|\rho_0\|_{L^1(\TT)},\beta\right).
$$
For nonnegative initial data the $L^1(\TT)$ norm is preserved, thus, 
$$
\|\rho(t)\|_{L^p(\TT)}\leq c\left(\|\rho_0\|_{L^1(\TT)},\beta,p\right),\;\;\forall 1\leq p\leq \infty.
$$ 
We need to study the evolution of the second derivative. We start with the aggregation terms. Using H\"{o}lder inequality, we get
\begin{multline*}
I_1=\int_\TT \pax^2\rho\pax^3(\rho\pax v)dx=\int_\TT\pax^2\rho\left(\frac{5}{2}\pax^2\rho(\rho-\rh)+3(\pax\rho)^2+\pax^2\rho\rho\right)dx\\
\leq \|\pax^2\rho\|_{L^2(\TT)}\left(\frac{7}{2}\|\rho-\rh\|_{L^\infty(\TT)}\|\pax^2\rho\|_{L^2(\TT)}+3\|\pax\rho\|_{L^4(\TT)}^2+\rh\|\pax^2\rho\|_{L^2(\TT)}\right).
\end{multline*}
Now we use \eqref{GN}-\eqref{S} and Poincar\'e inequality. We obtain
\begin{equation}\label{global1}
I_1\leq \left(\frac{7 C_S}{2}+9C_S\right)\|\pax^2\rho\|_{L^2(\TT)}^3+\rh\|\pax^2\rho\|_{L^2(\TT)}^2.
\end{equation}
We study the diffusion term 
\begin{multline*}
I_2=-\int_\TT \pax^2\rho\pax^2(\beta(\rho)\Lambda\rho)dx\\
=-\int_\TT\pax^2\rho\left(\beta''(\rho)(\pax\rho)^2\Lambda\rho+\beta'(\rho)\pax^2\rho\Lambda\rho
+2\beta'(\rho)\pax\rho\Lambda\pax\rho+\beta(\rho)\Lambda\pax^2\rho\right)dx.
\end{multline*}
Using \eqref{intporpart}, we get
\begin{multline*}
I_2
\leq C_{\beta''}\|\pax^2\rho\|_{L^2(\TT)}\|\pax\rho\|_{L^4(\TT)}^2\|\Lambda\rho\|_{L^\infty(\TT)}+C_{\beta'}\|\pax^2\rho\|_{L^2(\TT)}^2\|\Lambda\rho\|_{L^\infty(\TT)}\\
+2C_{\beta'}\|\Lambda\pax\rho\|_{L^2(\TT)}\|\pax\rho\|_{L^\infty(\TT)}\|\pax^2\rho\|_{L^2(\TT)}+\|\Lambda\beta(\rho)\|_{L^\infty(\TT)}\|\pax^2\rho\|_{L^2(\TT)}^2\\
-\nu\|\Lambda^{1/2}\pax^2\rho\|_{L^2(\TT)}^2,
\end{multline*}
with 
$$
C_{\beta^i}=\sup_{y\in \left[0,\|\rho(t)\|_{L^\infty(\TT)}\right]}\left|\frac{d^i\beta}{dy^i}(y)\right|.
$$
Since $\beta(y)$ is smooth and we have Proposition \ref{max}, these finite constants $C_{\beta^i}$ depend on $\beta$ and $\|\rho_0\|_{L^1(\TT)})$. Using the cancellation coming from the principal value integral, we get
$$
|\Lambda\rho|\leq\frac{\|\pax \rho\|_{C^{1/2}}}{2\pi}\int_\TT\frac{|y|^ {3/2}}{\sin^2\left(y/2\right)}dy\leq 6 C_S\|\pax^2\rho\|_{L^2(\TT)}.
$$
With this bound and the inequalities \eqref{GN}-\eqref{S} and \eqref{boundlambda}, we obtain
\begin{equation}\label{global2}
I_2\leq \left(18+\frac{\pi^2}{2} \right)C^2_SC_{\beta''}\|\pax^2\rho\|^4_{L^2(\TT)}+14 C_SC_{\beta'}\|\pax^2\rho\|^3_{L^2(\TT)}-\nu\|\Lambda^{1/2}\pax^2\rho\|_{L^2(\TT)}^2.
\end{equation}
The last term is the transport term with singular velocity:
\begin{multline*}
I_3=-\int_\TT \pax^2\rho\pax^2(\beta'(\rho)\pax \rho H\rho)dx\\
=-\int_\TT\pax^ 2\rho\left(\beta'''(\rho)(\pax\rho)^ 3H\rho+3\beta''(\rho)\pax\rho\pax^ 2\rho H\rho+2\beta''(\rho)(\pax\rho)^2\Lambda\rho\right.\\
\left.+2\beta'(\rho)\pax^ 2\rho\Lambda\rho+\beta'(\rho)\pax\rho\Lambda\pax\rho+\beta'(\rho)\pax^ 3\rho H\rho\right)dx.
\end{multline*}
With estimates that mimic the previous ones, we obtain
\begin{equation}\label{global3}
I_3\leq C_{\beta'''}C_S^3\|\pax^2\rho\|_{L^2(\TT)}^5+6C_S^2C_{\beta''}\|\pax^2\rho\|_{L^2(\TT)}^4
+19C_SC_{\beta'}\|\pax^2\rho\|_{L^2(\TT)}^3.
\end{equation}
Collecting all the estimates \eqref{global1}--\eqref{global3}, we get
\begin{multline*}
\frac{1}{2}\frac{d}{dt}\|\pax^ 2\rho\|^ 2_{L^2(\TT)}\leq C_{\beta'''}C_S^3\|\pax^2\rho\|_{L^2(\TT)}^5+\left(24+\frac{\pi^2}{2} \right)C^2_SC_{\beta''}\|\pax^2\rho\|_{L^2(\TT)}^4\\
+\left(\frac{25 C_S}{2}+33C_SC_{\beta'}\right)\|\pax^2\rho\|_{L^2(\TT)}^3+\rh\|\pax^2\rho\|_{L^2(\TT)}^2-\nu\|\Lambda^{1/2}\pax^2\rho\|_{L^2(\TT)}^2.
\end{multline*}
Using the fractional Poincar\'e inequality, we obtain
\begin{multline}\label{H2decay}
\frac{d}{dt}\|\pax^ 2\rho\|_{L^2(\TT)}\leq \left(C_{\beta'''}C_S^3\|\pax^2\rho\|_{L^2(\TT)}^3+\left(24+\frac{\pi^2}{2} \right)C^2_SC_{\beta''}\|\pax^2\rho\|_{L^2(\TT)}^2\right.\\
\left.+\left(\frac{25 C_S}{2}+33C_SC_{\beta'}\right)\|\pax^2\rho\|_{L^2(\TT)}+\rh-\nu\right)\|\pax^2\rho\|_{L^2(\TT)}.
\end{multline}
Now, if $\nu>\rh$ there exists an explicit (see \eqref{H2decay}) constant $\mathcal{C}=\mathcal{C}(\beta,\rh)$ such that, if the following inequality holds $\|\pax^2\rho_0\|_{L^2(\TT)}<\mathcal{C}$, we get
$$
\|\pax^2\rho(t)\|_{L^2(\TT)}\leq \|\pax^2\rho_0\|_{L^2(\TT)}\quad\forall 0\leq t\leq \tau
$$
where $\tau$ is the maximum lifespan of the solution. Using Proposition \ref{max}, we conclude
$$
\|\rho(t)\|_{H^2(\TT)}\leq C(\rho_0)\quad \forall 0\leq t\leq \tau,
$$
independent of $\tau$. Thus, by a standard continuation argument, we obtain the existence up to time $T$ for every $0<T<\infty$.
\end{proof}

\section{Global existence of weak solution for small $L^1$ initial data}\label{sec5}
In this section, we consider 
\begin{equation}\label{cbeta3}
\beta(\rho)=\rho+\nu, 
\end{equation}
with the constant $\nu>0$. Thus, our problem is
\begin{equation}\label{casoangel}
\left\{
\begin{array}{l}
\pat \rho=-\pax\left((\rho+\nu)H\rho\right)+\pax\left(\rho\pax v\right),\qquad x\in\TT, t>0,\\
\pax^2 v=\rho-\rh,
\end{array}
\right.
\end{equation} 
and an initial data $\rho_0\in L^\infty(\TT)\cap H^{1/2}(\TT)$.

We define our concept of weak solutions:
\begin{defi} $\rho(x,t)$ is a weak solution of \eqref{casoangel} if the following equality holds
\begin{multline*}
\int_0^T\int_\TT\pat\phi(x,t)\rho(x,t)dx dt+\int_\TT\rho_0(x)\phi(x,0)dx\\
=\int_0^T\int_\TT \pax\phi(x,t)\left[-(\rho+\nu)H\rho+\rho\pax v\right]dx dt,
\end{multline*}
for all $\phi(x,t)\in C^\infty_c\left([0,T),C^\infty(\TT)\right)$. If the previous condition holds for every $0<T<\infty$, $\rho$ is a global weak solution.
\end{defi}

First, recall some important results concerning fractional Sobolev spaces:
\begin{enumerate}
\item $H^{1/2}(\TT)$ is continuously embedded in $L^q(\TT)$ for every $q\in[1,\infty)$ (see Theorem 6.10 in \cite{Valdinoci1}).
\item $H^{1/2}(\TT)$ is compactly embedded in $L^q(\TT)$ for every $q\in[1,2]$ (see Theorem 7.1 in \cite{Valdinoci1} and Lemma 10 in \cite{Valdinoci2}). Moreover, $H^s(\TT)$ is compactly embedded in $L^q(\RR)$ for $0<s<1/2$ and $1\leq q<2/(1-2s)$ (see Corollary 7.2 in \cite{Valdinoci1}).
\end{enumerate}
We will use the Tricomi relation for periodic, mean zero functions:
\begin{equation*}
H\left(gHf+fHg\right)=HfHg-fg,
\end{equation*}
which, in the case $f=g$, reduces to
\begin{equation}\label{tricomi2}
2H\left(fHf\right)=(Hf)^2-f^2.
\end{equation}
\begin{teo}
Let $\rho_0\in L^\infty(\TT)\cap H^{1/2}(\TT)$ be a positive initial data and assume that 
$$
\|\rho_0\|_{L^1(\TT)}\leq\frac{2}{3}\nu.
$$
Then, there exist a unique solution of \eqref{casoangel} such that
$$
\rho(x,t)\in L^\infty\left([0,\infty),H^{1/2}(\TT)\cap L^\infty(\TT)\right)\cap C\left([0,\infty),L^2(\TT)\right).
$$
\end{teo}

\begin{proof}
\textbf{The regularized system:} The regularized system that we are considering is
\begin{equation}\label{casoangeleps}
\left\{
\begin{array}{l}
\pat \reps=-\pax\left((\reps+\nu)H\reps\right)+\pax\left(\reps\pax v\right)+\epsilon\pax^2\reps,\\
\pax^2 \veps=\reps-\rh,\\
\reps(x,0)=\jeps\rho_0(x),
\end{array}
\right.
\end{equation}
where $\mathcal{J}_\epsilon$ defined as in \eqref{epsi}. Notice that 
$$
\|\reps(0)\|_{L^\infty(\TT)}\leq\|\rho_0\|_{L^\infty(\TT)},\|\reps(0)\|_{L^1(\TT)}=\|\rho_0\|_{L^1(\TT)},\|\reps(0)\|_{H^{1/2}(\TT)}\leq\|\rho_0\|_{H^{1/2}(\TT)}.
$$
Moreover $\reps(x,0)\in H^s(\TT)$ for any $s>0$.

\textbf{The \emph{a priori} bounds:} Since $\beta(\rho)$ defined by (\ref{cbeta3}) satisfies the hypothesis in Proposition \ref{max}, we get 
$$
\|\reps(t)\|_{L^\infty(\TT)}\leq \max\{\|\reps(0)\|_{L^\infty(\TT)},2\pi\|\reps(0)\|_{L^1(\TT)}\}\leq C(\rho_0).
$$
Moreover, we obtain
\begin{equation}\label{Lpbound}
\|\reps(t)\|_{L^p(\TT)}\leq C(\rho_0,p).
\end{equation}
We study the evolution of the $H^{1/2}$ seminorm:
\begin{multline*}
\frac{1}{2}\frac{d}{dt}\|\Lambda^{1/2}\reps(t)\|^2_{L^2(\TT)}=\int_\TT\Lambda\reps\pat\reps dx
\\=-\|\sqrt{\reps}\Lambda\reps\|_{L^2(\TT)}^2-\nu\|\Lambda\reps\|_{L^2(\TT)}^2
-\rh\|\Lambda^{1/2}\reps\|_{L^2(\TT)}^2-\epsilon\|\pax\reps\|_{L^2(\TT)}^2\\-\int_\TT\Lambda\reps \pax\reps H\reps dx+\int_\TT\Lambda\reps \reps^2+\int_\TT\Lambda\reps\pax\reps\pax \veps.
\end{multline*}
Using \eqref{tricomi2}, we have
\begin{multline*}
I_1=-\int_\TT\Lambda\reps \pax\reps H\left(\reps-\rh\right) dx=\int_\TT (\reps-\rh)H\left(H\pax\reps \pax\reps\right)dx\\
=\frac{1}{2}\int_\TT\left(\reps-\rh\right)\left(\Lambda\reps\right)^2dx-\frac{1}{2}\int_\TT\left(\reps-\rh\right)\left(\pax\reps\right)^2dx\\
=\frac{1}{2}\int_\TT\reps\left(\Lambda\reps\right)^2dx-\frac{1}{2}\int_\TT\reps\left(\pax\reps\right)^2dx,
\end{multline*}
where in the last step we use \eqref{calderon}. We consider $\delta>0$ a positive number that will be fixed below. Then, we obtain 
$$
I_2=\int_\TT\Lambda\reps \reps^2\leq \|\Lambda\reps\|_{L^2(\TT)}\|\reps\|_{L^4(\TT)}^2\leq \delta\|\Lambda\reps\|^2_{L^2(\TT)}+\frac{c(\rho_0)}{\delta}.
$$
Notice that, using the equation of $\veps$ and its periodicity, we have
$$
\pax\veps(x,t)-\pax\veps(-\pi,t)=\int_{-\pi}^x\reps(y)-\rh dy,
$$
and, integrating by parts,
\begin{multline*}
0=\int_\TT\pax\veps(y,t)dy=2\pi\pax\veps(-\pi,t)-\int_\TT \pax^2\veps(y,t) ydy\\
=2\pi\pax\veps(-\pi,t)-\int_\TT \reps(y,t) ydy.
\end{multline*}
>From these two equalities we obtain
$$
\|\pax\veps\|_{L^\infty(\TT)}\leq |\pax \veps(-\pi,t)|+\|\rho_0\|_{L^1(\TT)}\leq \frac{3}{2}\|\rho_0\|_{L^1(\TT)}.
$$
The last integral is, using again \eqref{calderon},
$$
I_3=\int_\TT\Lambda\reps\pax\reps\pax \veps dx\leq\frac{3}{2}\|\rho_0\|_{L^1(\TT)}\|\Lambda\reps\|_{L^2(\TT)}^2. 
$$
Collecting all the estimates, we get
\begin{eqnarray*}
\frac{1}{2}\frac{d}{dt}\|\Lambda^{1/2}\reps\|^2_{L^2(\TT)}&\leq& -\frac{1}{2}\|\sqrt{\reps}\Lambda\reps\|_{L^2(\TT)}^2-\frac{1}{2}\|\sqrt{\reps}\pax\reps\|_{L^2(\TT)}^2-\rh\|\Lambda^{1/2}\reps\|_{L^2(\TT)}^2\\&&-\epsilon\|\pax\reps\|_{L^2(\TT)}^2
+\left(\delta+\frac{3}{2}\|\rho_0\|_{L^1(\TT)}-\nu\right)\|\Lambda\reps\|^2_{L^2(\TT)}+\frac{c(\rho_0)}{\delta}\\
&\leq& -\rh\left(\|\Lambda^{1/2}\reps\|_{L^2(\TT)}^2-\frac{c(\rho_0)}{\rh\delta}\right),
\end{eqnarray*}
if $\delta$ is taken sufficiently small. Using Gronwall inequality, we obtain
\begin{multline}\label{H1/2}
\|\Lambda^{1/2}\reps(t)\|_{L^2(\TT)}^2\leq\frac{c(\rho_0)}{\rh\delta}+\left(\|\Lambda^{1/2}\reps(0)\|_{L^2(\TT)}^2-\frac{c(\rho_0)}{\rh\delta}\right)e^{-2\rh t}\\\leq \frac{c(\rho_0)}{\rh\delta}+\left(\|\Lambda^{1/2}\rho_0\|_{L^2(\TT)}^2-\frac{c(\rho_0)}{\rh\delta}\right)e^{-2\rh t}.
\end{multline}

\textbf{Existence:} We study the evolution of $\|\reps(t)\|_{H^2}.$ First we deal with the nonlocal flux. The diffusive term can be bounded using \eqref{calderon}, \eqref{calderon2} and \eqref{GN} in the usual way
\begin{multline*}
A_1=-\int\pax^2\reps\pax^2\left(\reps\Lambda\reps\right)dx=\int_\TT\pax^3\reps
\left(\pax\reps\Lambda\reps+\reps\Lambda\pax\reps\right)dx\\
\leq\|\pax^3\reps\|_{L^2(\TT)}\left(\|\pax \reps\|_{L^4(\TT)}\|H\pax \reps\|_{L^4(\TT)}+\|H\pax^2\reps\|_{L^2(\TT)}\|\reps\|_{L^\infty(\TT)}\right)\\
\leq \|\pax^3\reps\|_{L^2(\TT)}\|\pax^2\reps\|_{L^2(\TT)}C(\rho_0).
\end{multline*}
To handle the transport term with singular velocity we need the Kato-Ponce inequality (see \cite{grafakos2013kato, katoponce})
\begin{equation}\label{KatoPonce}
\|\Lambda^s\left(fg\right)\|_{L^r}\leq C\left(\|g\|_{L^{p_1}}\|\Lambda^s f\|_{L^{p_2}}+\|\Lambda^s g\|_{L^{q_1}}\|f\|_{L^{q_2}}\right),
\end{equation}
where $s>0$ and
$$
\frac{1}{r}=\frac{1}{p_1}+\frac{1}{p_2}=\frac{1}{q_1}+\frac{1}{q_2}.
$$
Using \eqref{KatoPonce} in the case $p_1=p_2=q_1=q_2=2$, $r=s=1$
\begin{eqnarray*}
A_2&=&-\int\pax^2\reps\pax^2\left(\pax\reps H\reps\right)dx=\int_\TT\pax^3\reps
\left(\pax^2\reps H\reps+\pax\reps\Lambda\reps\right)dx\\
&\leq&-\frac{1}{2}\int_\TT\Lambda\left(\pax^2\reps\right)^2\reps dx+C(\rho_0)\|\pax^3\reps\|_{L^2(\TT)}\|\pax^2\reps\|_{L^2(\TT)}\\
&\leq& C(\rho_0) \|\pax^3\reps\|_{L^2(\TT)}\|\pax^2\reps\|_{L^2(\TT)}.
\end{eqnarray*}
The aggregation terms are:
\begin{eqnarray*}
A_3&=&\int\pax^2\reps\pax^2\left(\pax\reps\pax\veps\right)dx=-\int\pax^3\reps\left(\pax^ 2\reps\pax\veps+\pax\reps\left(\reps-\rh\right)\right)dx\\
&\leq& C(\rho_0)\|\pax^2\reps\|_{L^2(\TT)}\left(\|\pax^2\reps\|_{L^2(\TT)}+\|\pax^3\reps\|_{L^2(\TT)}\right),
\end{eqnarray*}
and
$$
A_4=\int\pax^2\reps\pax^2\left(\reps\left(\reps-\rh\right)\right)dx\leq C(\rho_0)\|\pax^2\reps\|_{L^2(\TT)}^2.
$$
Thus, using Young inequality, we obtain
$$
\frac{d}{dt}\|\pax^2\reps\|_{L^2(\TT)}\leq c(\epsilon,\rho_0)\|\pax^2\reps\|_{L^2(\TT)},
$$
and, using Gronwall inequality,
$$
\|\pax^2\reps(t)\|_{L^2(\TT)}\leq c_1(\epsilon,\rho_0)e^{c(\epsilon,\rho_0)t}.
$$
Since we have \eqref{Lpbound}, we have
$$
\|\reps(t)\|_{H^2(\TT)}\leq c(\epsilon,\rho_0,T),\;\;\forall\; T<\infty.
$$
With this estimate and following the classical technique, we obtain 
$$
\reps(x,t)\in C([0,\infty),H^2(\TT)),\;\;\forall\;\epsilon>0.
$$ 
\textbf{Compactness:} This step uses classical tools from functional analysis. Let $T>0$ be an arbitrary but finite final time. The estimate \eqref{Lpbound} gives us
$$
\sup_{t\in[0,T]}\|\reps(t)\|_{L^p(\TT)}\leq C(\rho_0,p),\;\forall\; 1\leq p\leq\infty.
$$
Thus, the family of approximate solutions remains uniformly bounded in the Bochner space $L^\infty\left([0,T],L^p(\TT)\right)$, $\forall\; 1\leq p\leq\infty$. Using \eqref{H1/2}, we have
$$
\sup_{t\in[0,T]}\|\reps(t)\|_{H^{1/2}(\TT)}\leq C(\rho_0).
$$
Using this two estimates, we get
\begin{equation}\label{Lptiempo}
\reps(t)\in L^p\left([0,T],H^{1/2}(\TT)\right)\cap L^p\left([0,T],L^{\infty}(\TT)\right), \;\;\forall\; 1\leq p\leq\infty.
\end{equation}
In particular, $\reps$ is uniformly bounded in the space $L^2\left([0,T],H^{1/2}(\TT)\right)$. Using the Banach-Alaoglu Theorem we obtain (picking a subsequence) the existence of $\rho$ such that
\begin{equation}\label{L2L2}
\reps(x,t)\rightharpoonup \rho(x,t)\in L^2\left([0,T],H^{1/2}(\TT)\right),
\end{equation} 
and, using \eqref{H1/2},
$$
\rho(x,t)\in L^\infty\left([0,T],H^{1/2}(\TT)\right).
$$
Using \eqref{Lpbound} and picking a subsequence if needed, we obtain 
$$
\reps(x,t)\stackrel{*}{\rightharpoonup}\rho(x,t)\in L^\infty\left([0,T],L^\infty(\TT)\right).
$$
We need to obtain some bound in $\pat \reps$ to obtain the compacity in some Bochner space. Given $f\in L^2(\TT)$, we take into account the norm
$$
\|f\|_{H^{-1}(\TT)}=\sup_{\substack{\psi\in H^1(\TT), \\ \|\psi\|_{H^1(\TT)}\leq 1}}\left|\int_{\TT}\psi(x)f(x)dx\right|.
$$
We consider the Banach space $H^{-1}(\TT)$ as the completion of $L^2(\TT)$ with this norm. We multiply the equation \eqref{casoangeleps} by $\psi\in H^1(\TT)$ and integrate to get
$$
\left|\int_\TT\pat\reps(x,t)\psi(x)dx\right|=\left|\int_\TT\left(\left(\reps+\nu\right) H\reps-\reps \pax\veps\right)\pax\psi dx\right|\leq C(\rho_0)\;\;\forall \psi\in H^1(\TT).
$$
Thus, we have
$$
\sup_{t\in[0,T]}\|\pat\reps(t)\|_{H^{-1}(\TT)}\leq C(\rho_0),
$$
and conclude 
$$
\pat \reps\in L^p\left([0,T],H^{-1}(\TT)\right)\;\;\forall\; 1\leq p \leq \infty.
$$
We use the classical Aubin-Lions Lemma to obtain compactness (see Corollary 4, Section 8 in \cite{simon1986compact}). Let us restate this result: given three spaces $X\subset B\subset Y$, such that the embedding $X\subset B$ is compact and the embedding $B\in Y$ is continuous, we consider a sequence $f_n$ satisfying 
\begin{enumerate}
\item $f_n$ is uniformly bounded in $L^\infty([0,T],X)$,
\item $\pat f_n$ is uniformly bounded in $L^r([0,T],Y)$ where $r>1$.
\end{enumerate}
Then this sequence is relatively compact in $C([0,T],B)$. Thus, we take $X=H^{1/2}(\TT), B= L^2(\TT)$ and $Y=H^{-1}(\TT)$ and with this strong convergence, we get 
$$
\sup_{t\in[0,T]}\|H\reps(t)-H\rho(t)\|_{L^2(\TT)}\leq \sup_{t\in[0,T]}\|\reps(t)-\rho(t)\|_{L^2(\TT)}\rightarrow 0
$$
Thus, picking a subsequence, $\reps(x,t)\rightarrow \rho(x,t)$ and $H\reps(x,t)\rightarrow H\rho(x,t)$ almost everywhere.

\begin{figure}[t]
		\begin{center}
		\includegraphics[scale=0.45]{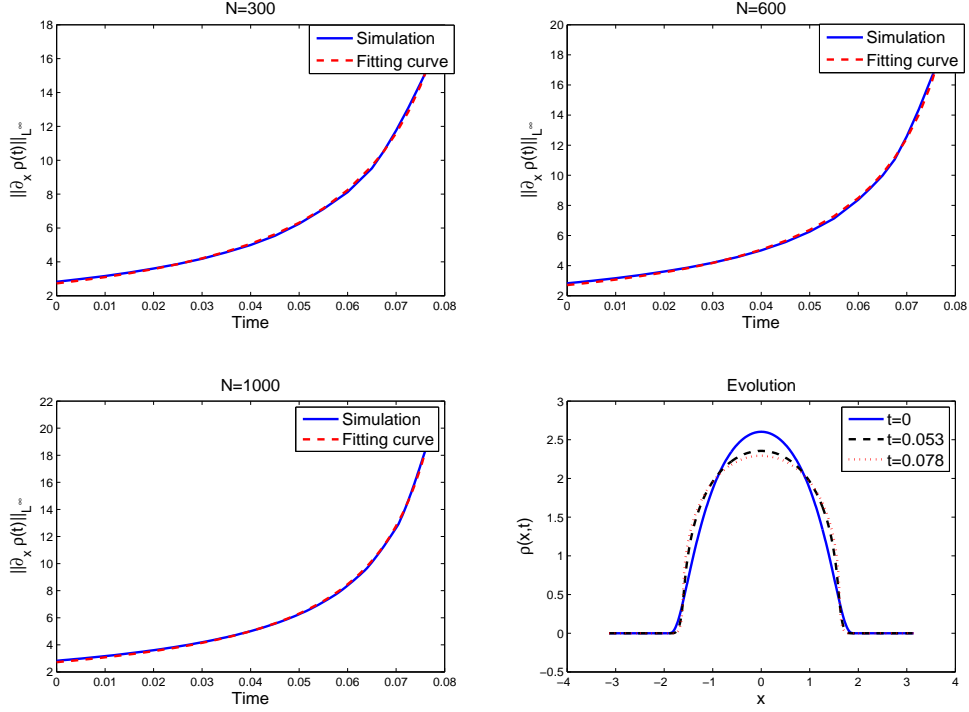} 
		\end{center}
		\caption{Case 1. $\beta(\rho)=\rho^2$.}		
		\label{Case1}
\end{figure}

Now, we need to pass to the limit in the weak formulation
\begin{multline*}
\int_0^T\int_\TT\pat\phi(x,t)\reps(x,t)dx dt+\int_\TT\reps(x,0)\phi(x,0)dx\\
=\int_0^T\int_\TT \pax\phi(x,t)\left[-(\reps+\nu)H\reps+\reps\pax \veps\right]dx dt-\epsilon\int_\TT\int_\TT\reps(x,t)\pax^ 2\phi(x,t)dxdt.
\end{multline*}
Using \eqref{Lpbound} and \eqref{L2L2}, we obtain the convergence of the linear terms. The nonlinear terms can be handled as follows
\begin{multline*}
\left|\int_0^T\int_\TT\pax\phi(x,t)\left(\reps H\reps-\rho H\reps+\rho H\reps-\rho H\rho\right)dxdt\right|\\
\leq T\|\pax \phi\|_{L^\infty(\TT\times[0,T])}\sup_{t\in[0,T]}\|\reps(t)-\rho(t)\|_{L^2(\TT)}\sup_{t\in[0,T]}(\|H\reps(t)\|_{L^2(\TT)}+
\|\reps(t)\|_{L^2(\TT)})\\
\leq C(\phi,\rho_0,T)\sup_{t\in[0,T]}\|\reps(t)-\rho(t)\|_{L^2(\TT)}\rightarrow0,
\end{multline*}
as $\epsilon\to 0$. Now, since $v$ is defined by \eqref{eq2} (we can since $\rho(t)\in L^2(\TT)$ for all times),
\begin{multline*}
\left|\int_0^T\int_\TT\pax\phi(x,t)\left(\reps \pax\veps-\rho \pax\veps+\rho \pax \veps-\rho \pax v\right)dxdt\right|\\
\leq T\|\pax \phi\|_{L^\infty(\TT\times[0,T])}\left(\sup_{t\in[0,T]}\|\reps(t)-\rho(t)\|_{L^2(\TT)}\sup_{t\in[0,T]}\|\pax\veps(t)\|_{L^2(\TT)}\right.\\
\qquad\qquad\qquad\left.+\sup_{t\in[0,T]}\|\pax\veps(x,t)-\pax v(x,t)\|_{L^2(\TT)}\sup_{t\in[0,T]}\|\rho(t)\|_{L^2(\TT)}\right)\\
\leq C(\phi,\rho_0,T)\sup_{t\in[0,T]}\|\reps(t)-\rho(t)\|_{L^2(\TT)}\rightarrow0.
\end{multline*}
We conclude the proof.
\end{proof}

\section{Numerical simulations}\label{sec6}
To better understand the role of $\beta(\rho)$, we perform some numerical simulations. We denote $N$ the number of spatial grid points and we approximate our solution by a cubic spline passing through these nodes. Then, we compute (using the function \texttt{quadl} in Matlab) the Hilbert transform using Taylor series and the cancellation coming from the principal value integration to avoid the singularity of the integral. Once that we compute $H\rho$, multiplying by $\beta(\rho)$ and taking the derivative, we have the nonlocal flux. The Poisson equation is solved using finite differences. This ends with the spatial part in a straightforward way. We advance in time with the Runge-Kutta-Fehlberg-45 scheme with tolerance $10^{-8}$.

We consider the same initial data
$$
\rho(x,0)=\left(\frac{\int_{-\pi}^\pi e^{-\frac{1}{1-(s/2)^2}}ds}{2\pi}\right)^{-1}e^{-\frac{1}{1-(x/2)^2}}
$$
in all simulations. In the first case we take $\beta(\rho)$ a convex function, while in the second simulation we consider a concave one. 

\begin{figure}[t]
		\begin{center}
		\includegraphics[scale=0.45]{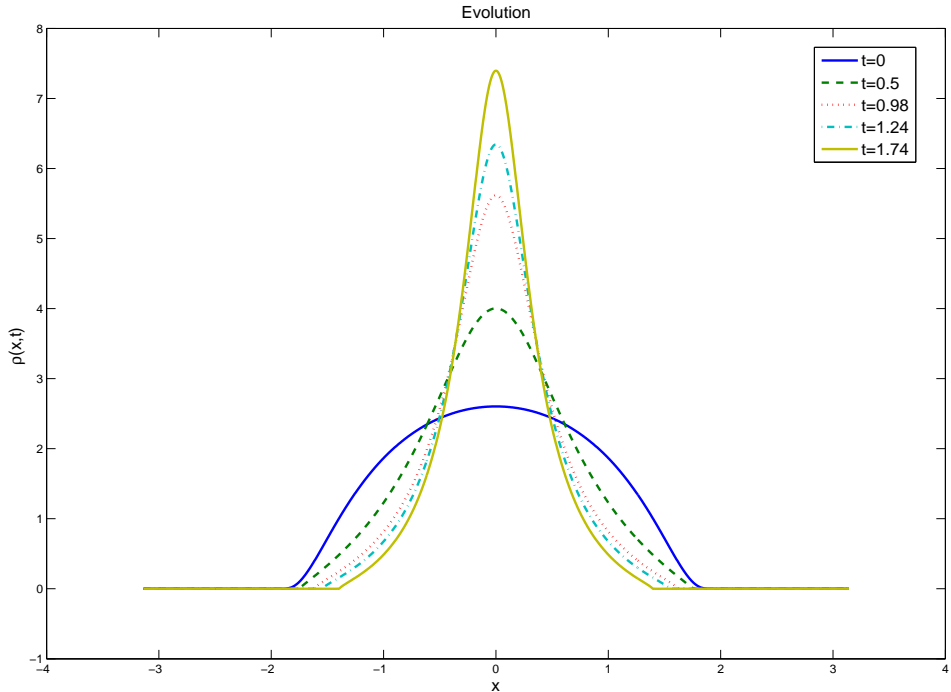} 
		\end{center}
		\caption{Case 2. $\beta(\rho)=log(1+\rho)$.}		
		\label{Case2}
\end{figure}

\textbf{Case 1:} We consider $\beta(\rho)=\rho^2$. The results are contained in Figure \ref{Case1}. Notice that the first derivative appears to blow up even if we refine $N$. We conjecture that $\|\pax \rho(t)\|_{L^\infty(\TT)}$ behaves like 
$$
\frac{C}{\left(T-t\right)^a}.
$$
Using least squares, we approximate these parameters for different values of $N$, in particular, $N=300$, 600, 1000, to get
$$
C=0.147126,\; T=0.093494,\; a=1.191234.
$$
With these constants we believe that the blow up occurs.

\textbf{Case 2:} We consider $\beta(\rho)=\log(1+\rho)$. Now the diffusion can not prevent that $\|\rho(t)\|_{L^\infty(\TT)}$ grows (even if we know that it is uniformly bounded for all times) and we obtain a very different profile (see Figure \ref{Case2}). Here, even if the $\|\rho(t)\|_{C^2}$ increases, there is no evidence of finite time blow up.

\end{document}